\documentclass[review]{elsarticle}

\makeatletter
\def\ps@pprintTitle{%
 \let\@oddhead\@empty
 \let\@evenhead\@empty
 \def\@oddfoot{\hfill\today}%
 \let\@evenfoot\@oddfoot}
\makeatother

\usepackage{hyperref}
\makeatletter
\providecommand{\doi}[1]{%
  \begingroup
    \let\bibinfo\@secondoftwo
    \urlstyle{rm}%
    \href{http://dx.doi.org/#1}{%
      doi:\discretionary{}{}{}%
      \nolinkurl{#1}%
    }%
  \endgroup
}
\makeatother

\usepackage{enumitem}
\usepackage{amsmath,amsthm,amssymb,bm}
\usepackage{graphicx}
\usepackage{tikz}
\usepackage{tikz-qtree}
\usepackage{todonotes}
\usepackage{cleveref}
\usepackage{algpseudocode,algorithm}
\algnewcommand\algorithmicinput{\textbf{Input:}}
\algnewcommand\Input{\item[\algorithmicinput]}
\algnewcommand\algorithmicoutput{\textbf{Output:}}
\algnewcommand\Output{\item[\algorithmicoutput]}
\usepackage{caption}

\newtheorem{thm}{Theorem}
\newtheorem{lemma}[thm]{Lemma}
\newtheorem{proposition}[thm]{Proposition}

\theoremstyle{definition}

\newtheorem{expl}{Example}
\newtheorem{remark}{Remark}

\newcommand{\Z}{\mathbb{Z}}

\newcommand{\Q}{\mathbb{Q}}
\newcommand{\C}{\mathbb{C}}
\newcommand{\ann}{\textrm{Ann}}
\newcommand{\Span}[1]{\mathrm{Span}_{\Q}\{{#1}\}}
\newcommand{\Spanc}[1]{\mathrm{Span}_{\C}\{{#1}\}}

\newcommand{\diag}{\mathrm{diag}}

\begin{document}

\begin{frontmatter}
  \title{Fast Algorithms for Computing Eigenvectors of Matrices via
    Pseudo Annihilating Polynomials\tnoteref{t1}} 
  \author[niigata]{Shinichi Tajima}
  \ead{tajima@emeritus.niigata-u.ac.jp}

  \author[kanazawa]{Katsuyoshi Ohara}
  \ead{ohara@se.kanazawa-u.ac.jp}
  \ead[url]{http://air.s.kanazawa-u.ac.jp/~ohara/}

  \author[tsukuba]{Akira Terui\corref{corauthor}}
  \cortext[corauthor]{Corresponding author}
  \ead{terui@math.tsukuba.ac.jp}
  \ead[url]{https://researchmap.jp/aterui}

  \address[niigata]{Graduate School of Science and Technology, Niigata
    University, Niigata 950-2181, Japan}
  \address[kanazawa]{Faculty of Mathematics and Physics, Kanazawa
    University, Kanazawa 920-1192, Japan}
  \address[tsukuba]{Faculty of Pure and Applied Sciences, University
    of Tsukuba, Tsukuba 305-8571, Japan}

  \tnotetext[t1]{This work has been partly supported by JSPS KAKENHI
    Grant Numbers JP15KT0102, JP18K03320, JP16K05035, and by the
    Research Institute for Mathematical Sciences, a Joint
    Usage/Research Center located in Kyoto University.}

  \begin{abstract}
    An efficient algorithm for computing eigenvectors of a matrix of
    integers by exact computation is proposed.  The components of
    calculated eigenvectors are expressed as polynomials in the
    eigenvalue to which the eigenvector is associated, as a
    variable.  The algorithm, in principle, utilizes the minimal
    annihilating polynomials for eliminating redundant
    calculations. Furthermore, in the actual computation, the
    algorithm computes candidates of eigenvectors by utilizing
    \emph{pseudo} annihilating polynomials and verifies their
    correctness. The experimental results show that our algorithms
    have better performance compared to conventional methods.
\end{abstract}

\end{frontmatter}

\newpage
\section{Introduction}
\label{sec:intro}

Exact linear algebra plays important roles in many fields of
mathematics and sciences. In recent years, this area has been
extensively studied and new algorithms have been
proposed for various types of computations,
such as computing canonical forms of matrices
(\cite{aug-cam1997},
\cite{dum-sau-vil2001},
\cite{mor2004},
\cite{per-ste2010},
\cite{sau-zhe2004},
\cite{sto2001},
\cite{sto-lab1996}),
the
characteristic or the minimal polynomial of a matrix
(\cite{dum-per-wan2005}, \cite{neu-pra2008}), LU and other
decompositions and/or solving a system of linear equations
(\cite{bos-jea-sch2008},
\cite{ebe-gie-gio-sto-vil2006},
\cite{jea-per-sto2013},
\cite{may-sau-wan2007},
\cite{sau-woo-you2011}),
 and several software have been developed
(\cite{alb2012},
\cite{che-sto2005},
\cite{linbox2002},
\cite{dum-gau-per2002},
\cite{dum2004}).


We have proposed, in the context of symbolic computation, a series of
algorithms on eigenproblems
including computation of (generalized) eigendecomposition and
spectral decomposition (\cite{oha-taj2009}).  In this paper, we propose
an effective method for computing eigenvectors of matrices of
integers or rational numbers.

Let $\lambda$ be an eigenvalue of a
matrix. In a conventional method of computing eigenvectors, the
eigenvector associated to $\lambda$ is
simply computed by solving a system of linear equations.
However, the method has a drawback that, if $\lambda$ is an
algebraic number, it uses solving a system of linear
equations with algebraic number arithmetic for computing the
eigenvector, which is inefficient.

In the proposed method, the components of eigenvectors are expressed
as polynomials in eigenvalues to which the eigenvector is
associated, as a variable.  Furthermore, in the case that $\lambda$
is an algebraic number and the geometric multiplicity of
$\lambda$ is equal to its algebraic multiplicity, it is sufficient to
compute just 
the algebraic multiplicity of $\lambda$ of eigenvectors for expressing
all the eigenvectors associated to \emph{all} the conjugates of
$\lambda$. A method for computing eigenvectors in this form has been
proposed by Takeshima and Yokoyama (\cite{tak-yok1990}) in the 1990s
by using the 
Frobenius normal form of $A$, and it has been extended by Moritsugu
and Kuriyama (\cite{mor-kur2001}) for the case that the Frobenius normal
form has multiple companion blocks and for computing generalized
eigenvectors.  In contrast, our approach is based on the concept of
the minimal annihilating polynomials (\cite{taj-oha-ter2018}) and the
Krylov vector spaces. We show that, with the
use of minimal annihilating polynomials, eigenvectors are computed
in an effective manner without solving a system of linear equations. 
Furthermore, the proposed method does not require computation of
canonical form of matrices.

We propose algorithms for computing eigenvectors under the
assumption that the geometric multiplicity of the eigenvalue is equal
to its algebraic multiplicity.  The resulting
algorithms have following features.  First, \emph{pseudo} minimal annihilating
polynomials are used for faster computation of eigenvectors. Second,
computation of a candidate of eigenvector is 
completed almost simultaneously as verification of pseudo annihilating
polynomial. Notably, the Horner's rule for matrices and vectors
is used in an effective manner for fast evaluations.

This paper is organized as follows.  In Section~\ref{sec:prem}, we
recall the notion of minimal annihilating polynomial and other
necessary concepts.  In 
Section~\ref{sec:eigenvect-multiplicity=1}, we describe a main idea of
an algorithm for 
computing eigenvectors just for the case that the algebraic
multiplicity of the 
eigenvalue is equal to $1$. In
Section~\ref{sec:sec:eigenvect-multiplicity>1}, we give, by using
Krylov vector spaces, an algorithm
for computing eigenvectors in the case that the algebraic multiplicity of
the eigenvalue is greater than $1$. In
Section~\ref{sec:eigenvector-upap}, we introduce the concept of 
pseudo annihilating polynomial and present algorithms for computing
eigenvectors using the pseudo annihilating polynomials.
In Section~\ref{sec:exp},
experimental results for the proposed algorithms are shown.

\section{Preliminaries}
\label{sec:prem}

Let $A$ be a $n\times n$ matrix over rational numbers, $\chi_A(\lambda)$
the characteristic polynomial of $A$, and $E$ the identity matrix
of dimension $n$. Assume that the 
irreducible factorization 
\begin{equation}
  \label{eq:charpol}
  \chi_A(\lambda)=f_1(\lambda)^{m_1}f_2(\lambda)^{m_2}\cdots
  f_q(\lambda)^{m_q}
\end{equation}
of $\chi_A(\lambda)$ is given, where $f_p(\lambda)\in\Q[\lambda]$,
$p=1,2,\ldots,q$. 


\subsection{The minimal annihilating polynomial}
\label{sec:map}

Let $\bm{v}$ be a non-zero vector in $\Q^n$. The monic generator of an
ideal $\ann_{\Q[\lambda]}(A,\bm{v})$ defined to be
\begin{equation}
  \label{eq:annideal}
  \ann_{\Q[\lambda]}(A,\bm{v}) = \{P(\lambda)\in \Q[\lambda]\mid
  P(A)\bm{v}=\bm{0}\},
\end{equation}
is called
\emph{the minimal annihilating polynomial} of $\bm{v}$ with respect to
$A$.
For $j\in J:=\{1,2,\ldots,n\}$, let
$\bm{e}_j={}^t(0,\ldots,0,1,0,\ldots,0)$ be the $n$ dimensional
standard unit vector and let $\pi_{A,j}(\lambda)$ denote the 
minimal annihilating polynomial of $\bm{e}_j$ with respect to $A$. 

Let
\begin{equation}
  \label{eq:map}
  \begin{split}
    \pi_{A,j}(\lambda)
    &=f_1(\lambda)^{l_{1,j}}f_2(\lambda)^{l_{2,j}}\cdots
    f_p(\lambda)^{l_{p,j}}\cdots f_q(\lambda)^{l_{q,j}},
    \\
    &\qquad 0\le l_{p,j}\le m_p,\quad j\in J,
  \end{split}
\end{equation}
be the irreducible factorization of $\pi_{A,j}(\lambda)$.

Let $g_{p,j}(\lambda)$ denote the cofactor in $\pi_{A,j}(\lambda)$ of
the eigenfactor $f_p(\lambda)$ defined to be
\begin{equation}
  \label{eq:gpj}
  g_{p,j}(\lambda)
  =f_1(\lambda)^{l_{1,j}}\cdots f_{p-1}(\lambda)^{l_{p-1,j}}
  f_{p+1}(\lambda)^{l_{p+1,j}}\cdots f_q(\lambda)^{l_{q,j}}.
\end{equation}

\subsection{Horner's rule for matrix polynomials}
\label{sec:horner}

Let $f(\lambda)$ be a polynomial in $\Q[\lambda]$ of degree $d$:
\begin{equation}
  \label{eq:f}
  f(\lambda)=a_d\lambda^d+a_{d-1}\lambda^{d-1}+\cdots+a_0\lambda^0,
\end{equation}
with $a_d\ne 0$. Define $\psi_f(x,y)$ as
\begin{equation}
  \label{eq:psi-p}
  \psi_f(x,y)=\frac{f(x)-f(y)}{x-y}.
\end{equation}
Since, $\psi_f(x,y)$ is the quotient of $f(x)$ on division by $x-y$,
the coefficients $c_i\in\Q[x]$, $i=d-1,d-2,\ldots,0$, of the expansion
\begin{equation}
  \label{eq:psi-f}
  \psi_f(x,y)=c_{d-1} y^{d-1}+c_1y^{d-2}+\cdots+c_{1}y+c_{0},
\end{equation}
of $\psi(x,y)$ with respect to $y$ satisfy the following recursion
relations:
\begin{equation}
  \label{eq:psi-horner}
  c_{d-1}=a_d,\quad c_{d-1-j}=c_{d-j}x+a_{d-j} \quad (j=1,\ldots,d-1).
\end{equation}

Let $\bm{v}\in\Q^n$. Then, the vector $f(A)\bm{v}$ and the coefficient
vectors $\bm{c}_i$, $i=d-1,d-2,\ldots,0$ are calculated by the
Horner's rule with multiplication of $\bm{v}$ from the right as
\begin{equation}
  \label{eq:fav}
  \begin{split}
    f(A)\bm{v} &= (a_dA^d+a_{d-1}A^{d-1}+\cdots+a_0E)\bm{v} \\
    &= A(\cdots A(A(a_n(A\bm{v}) +
    a_{d-1}\bm{v})+a_{d-2}\bm{v})\cdots)+a_0\bm{v}, \\
    \psi_f(A,\lambda E)\bm{v} & =\lambda^{d-1} \bm{c}_{d-1}+\lambda^{d-2}
    \bm{c}_{d-2}+\cdots+\lambda \bm{c}_{1}+\bm{c}_{0}, \\ 
    \bm{c}_{d-1}&=a_d\bm{v},\quad
    \bm{c}_{d-1-j}=A\bm{c}_{d-j}+a_{d-j}\bm{v} \quad (j=1,\ldots,d-1), 
  \end{split}
\end{equation}
respectively. Thus, total cost is bounded by $O(n^2)$ and
$O(n^2(d-1))$, respectively.

Notice that, $f(A)\bm{v}=A\bm{c}_0+a_0\bm{v}$ holds. This relation
will be used in Section~\ref{sec:eigenvector-upap}.

\begin{lemma}
  \label{lem:eigenvector-single}
  Let $\bm{u}\in\Q^n$ be a non-zero vector and let $f(\lambda)$ be the
  minimal annihilating polynomial of $\bm{u}$ with respect to $A$. Let
  $\varphi(\lambda)=\psi_f(A,\lambda E)\bm{u}$. Let $\alpha$ be a root
  of $f(\lambda)$. Then, $\varphi(\alpha)$ is an eigenvector of $A$
  associated to the eigenvalue $\alpha$.
\end{lemma}
\begin{proof}
  It follows immediately from $(x-y)\psi_f(x,y)=f(x)-f(y)$ that
  \[
  (A-\lambda E)\varphi(\lambda)=(f(A)-f(\lambda E))\bm{u}
  =-f(\lambda)\bm{u}.
  \]
  Therefore, 
  $(A-\alpha E)\varphi(\alpha)=\bm{0}$ holds.
  Since $f$ is the minimal annihilating polynomial and
  $\deg(\psi_f(x,y))<\deg(f(\lambda))$, we have
  $\varphi(\alpha)\ne\bm{0}$. This completes the proof. 
\end{proof}

Since $f$ is a factor of the characteristic polynomial
$\chi_A(\lambda)$, we call $\varphi(\lambda)$, an eigenvector
associated to the eigenfactor $f$, for simplicity.

Let us emphasize the fact that the eigenvector $\varphi(\lambda)$
introduced above represents all the eigenvectors
$\varphi(\alpha_1),\varphi(\alpha_2),\ldots,\varphi(\alpha_d)$
associated to the eigenvalues $\alpha_1,\alpha_2,\ldots,\alpha_d$ of
$A$ satisfying $f(\lambda)=0$.

\section{Main ideas}
\label{sec:eigenvect-multiplicity=1}

In this section, we show basic ideas of our approach for computing
eigenvectors. For this aim, we consider the simplest case where
algebraic multiplicity $m_p$ of an eigenfactor $f_p$ is equal to
one. We present a prototype of our method for illustration.

Assume that $m_p=1$ and all the minimal
annihilating polynomials \linebreak
$\pi_{A,1}(\lambda),\pi_{A,2}(\lambda),\ldots,\pi_{A,n}(\lambda)$ are
given.

Let 
\begin{equation}
  \label{eq:j1}
  J_0=\{j\in J\mid l_{p,j}=0\}, \quad J_1=\{j\in J\mid l_{p,j}=1\}.  
\end{equation}
Then, for $j\in J_1$, we have
$\pi_{A,j}(\lambda)=f_p(\lambda)g_{p,j}(\lambda)$, where
$g_{p,j}(\lambda)$ is the cofactor in $\pi_{A,j}$ of the eigenfactor
$f_p(\lambda)$. Now consider the vector
\begin{equation}
  \label{eq:vpj}
  \bm{v}_{p,j}=g_{p,j}(A)\bm{e}_j,
\end{equation}
for $j\in J_1$. Then, since
$f_p(\lambda)$ is the minimal annihilating polynomial of the non-zero
vector $\bm{v}_{p,j}$, $\varphi_j(\lambda)$ defined to be
\begin{equation}
  \label{eq:phij}
  \varphi_j(\lambda)=\psi_p(A,\lambda E)\bm{v}_{p,j},
\end{equation}
is an eigenvector associate with
the eigenfactor $f_p(\lambda)$, where $\psi_p(x,y)=\psi_{f_p}(x,y)$. 

The argument above leads a prototype for computing eigenvectors as
follows. 

\noindent{\hrulefill}
  \begin{algorithmic}[1]
    \Input{
      $A\in\Q^{n\times n}$; $f_p(\lambda)\in\Q[\lambda]$: an eigenfactor of
      $A$ with $m_p=1$; $J_1\subset J$; $\{g_{p,j}(\lambda)\mid j\in
      J_1\}$: a set of cofactors, defined as in \cref{eq:gpj};
    }
    \Output{
      $\varphi(\lambda)$: an eigenvector of $A$ associated to the
      root of $f_p(\lambda)=0$;
    }
    \State Select $j\in J_1$;
    \State $\bm{v}_{p,j}\gets g_{p,j}(A)\bm{e}_j$ with the
    Horner's rule (\cref{eq:fav}); \label{alg:eigenvector:gjp}
    \State $\varphi(\lambda)\gets\psi_p(A,\lambda E)\; \bm{v}_{p,j}$
    with the Horner's rule (\cref{eq:psi-horner}); \label{alg:eigenvector:rho}
    \State \Return $\varphi(\lambda)$;
  \end{algorithmic}
  \hrulefill

\begin{expl}
  \label{expl:eigenvector}
  Let
  \[
  A=
  \begin{pmatrix}
    -3 & -3 & -4 & 2 & 1 \\
    -114 & 56 & 12 & 6 & -3 \\
    330 & -179 & -50 & -11 & 12 \\
    423 & -255 & -88 & -4 & 22 \\
    -303 & 3 & -79 & 60 & 5 
  \end{pmatrix}
  .
  \]
  The characteristic polynomial $\chi_A(\lambda)$ and the unit
  minimal annihilating polynomial $\pi_{A,j}(\lambda)$,
  $j=1,2,\ldots,5$ are
  \begin{align*}
    \chi_A(\lambda) &= f_1(\lambda)f_2(\lambda),
    \\
    \pi_{A,3}(\lambda) &= f_2(\lambda),
    \pi_{A,1}(\lambda)=\pi_{A,2}(\lambda)=\pi_{A,4}(\lambda)=\pi_{A,5}(\lambda)
    =f_1(\lambda)f_2(\lambda), 
  \end{align*}
  where $f_1(\lambda)=\lambda^2+\lambda+12$,
  $f_2(\lambda)=\lambda^3-5\lambda^2-60\lambda-41$. 
  Let us compute the eigenvector $\varphi(\lambda)$ associated to
  the eigenfactor $f_2(\lambda)$, by using
  $\psi_2(x,y)=y^2+(x-5)y+x^2-5x-60$. Since $J_1=\{1,2,3,4,5\}$, any 
  vector from
  $\bm{v}_{2,1},\bm{v}_{2,2},\bm{v}_{2,3},\bm{v}_{2,4},\bm{v}_{2,5}$
  can be used. Here we take, for instance, two cases:

  \begin{enumerate}
  \item Computing $\varphi(\lambda)$ using $\bm{v}_{2,3}$: since
    $g_{2,3}(\lambda)=1$, 
    $\bm{v}_{2,3}=\bm{e}_3$. The eigenvector
    $\varphi(\lambda)$ is computed as
    \begin{align}
      \varphi(\lambda) &= \psi_2(A,\lambda E)\bm{e}_3=
      \{\lambda^2 E + \lambda(A-5E) + (A^2-5A-60E)\}\bm{e}_3
      \nonumber\\
      &=
      \lambda^2\bm{e}_3+\lambda(A\bm{e}_3-5\bm{e}_3)
      +(A(A\bm{e}_3-5\bm{e}_3)-60\bm{e}_3)
      \nonumber\\
      &=
      {}^t(0,0,1,0,0)\lambda^2+{}^t(-4,12,-55,-88,-79)\lambda
      \nonumber\\
      &\qquad +{}^t(-59,177,-758,-1298,-82)
      .
      \label{eq:ex1-v1}
    \end{align}
  \item Computing $\varphi(\lambda)$ using $\bm{v}_{2,1}$: since
    $g_{2,1}(\lambda)=f_1(\lambda)$, 
    \begin{equation}
      \label{eq:ex1-v21}
      \bm{v}_{2,1}=f_1(A)\bm{e}_1={}^t(-417,1251,-5043,-9174,-1941).      
    \end{equation}
    The eigenvector $\varphi(\lambda)$ is computed as
    \begin{align}
      \label{eq:ex1-v2}
      \varphi(\lambda) &= \psi_2(A,\lambda E)\bm{v}_{2,1} 
      =\{\lambda^2 E + \lambda(A-5E) + (A^2-5A-60E)\}\bm{v}_{2,1}
      \nonumber\\
      &=
      \lambda^2\bm{v}_{2,1}+\lambda(A\bm{v}_{2,1}-5\bm{v}_{2,1})
      +(A(A\bm{v}_{2,1}-5\bm{v}_{2,1})-60\bm{v}_{2,1})
      \nonumber\\
      &=
      {}^t(-417,1251,-5043,-9174,-1941)\lambda^2 \nonumber\\
      &\qquad +{}^t(-534,1602,-6552,-11748,-21939)\lambda
      \nonumber\\
      &\qquad +{}^t(2589,-7767,33162,56958,-13899).
    \end{align}
    Now consider the Krylov vector space
    $L_A(\bm{v}_{2,1})=\Span{\bm{v}_{2,1},A\bm{v}_{2,1},A^2\bm{v}_{2,1}}$
    generated by $\bm{v}_{2,1}$. From
    \begin{align*}
      \bm{v}_{2,1} &= {}^t(-417,1251,-5043,-9174,-1941), \\
      A\bm{v}_{2,1} &= {}^t(-2619,7857,-31767,-57618,-31644), \\
      A^2\bm{v}_{2,1} &= {}^t(-35526,106578,-428253,-781572,-288579),
    \end{align*}
    we have
    \[
    L_A(\bm{v}_{2,1})= \Span{{}^t(1,-3,0,22,0), \bm{e}_3, \bm{e}_5}.
    \]
    Therefore, the vector 
    \begin{multline}
      \label{eq:ex1-v22}
      \psi_2(A,\lambda E)\bm{e}_5=
      {}^t(0,0,0,0,1)\lambda^2
      +
      {}^t(1,-3,12,22,0)\lambda
      \\
      +
      {}^t(2,-6,25,44,0),
    \end{multline}
    constructed from the last basis vector $\bm{e}_5={}^t(0,0,0,0,1)$
    in $V=L_A(\bm{v}_{2,1})$ 
    is also an eigenvector associated to the eigenfactor
    $f_2(\lambda)$. 
  \end{enumerate}
  Notice that $\psi_2(A,\lambda E)\bm{e}_3$, $\psi_2(A,\lambda
  E)\bm{e}_5$ have simpler expression than $\psi_2(A,\lambda
  E)\bm{v}_{2,1}$. Notice also that, in all cases, the leading
  coefficient vector in $\varphi(\lambda)=\psi_2(A,\lambda E)\bm{u}$
  is equal to $\bm{u}$ itself.
  Note also that, if we consider
  $V=\Span{\bm{v}_{2,1},\bm{v}_{2,2},\ldots,\bm{v}_{2,5}}$, we also
  have
  \[
  V=\Span{{}^t(1,-3,0,22,0), \bm{e}_3,\bm{e}_5}.
  \]
\end{expl}

Let us turn back to the case $m_p=1$. Let
\[
V=\Span{\bm{v}_{p,j}\mid
  \bm{v}_{p,j}=g_{p,j}(A)\bm{e}_j,\ j\in J_1}.
\]
Then,
\[
\varphi(\lambda)=\psi_p(A,\lambda E)\bm{u},
\]
constructed from any non-zero vector $\bm{u}$ in $V$ gives rise to an
eigenvector associated to the eigenfactor $f_p(\lambda)$. In other
words, every non-zero vector $\bm{u}\in V$ has the same amount of
information on eigenspaces associated to the eigenfactor
$f_p(\lambda)$. In fact, if we consider the Krylov vector space
$L_A(\bm{u})$ defined to be
\begin{equation}
  \label{eq:lau}
  L_A(\bm{u})=\Span{\bm{u},A\bm{u},A^2\bm{u},\ldots,A^{d_p-1}\bm{u}},
\end{equation}
for $\bm{u}\in V$, with $d_p=\deg(f_p)$, we have $V=L_A(\bm{u})$.

This observation above yields the following two different strategies
for computing eigenvectors for the case $m_p=1$.
\begin{enumerate}[label=(\alph{enumi})]
\item The prototype method presented above requires
  $O(n^2\deg(\pi_{A,j}))$ operations for computing
  $\varphi(\lambda)$.  This suggests that, if one wants to obtain the
  eigenvector as quickly as possible, one should select the unit
  minimal annihilating polynomial $\pi_{A,j}(\lambda)$ of 1) smaller
  degree, or 2) if there are several ones of the same degree, select
  one with coefficients of smaller magnitudes, to reduce the amount of
  computation.
\item Recall the fact that the leading coefficient vector in
  $\psi_p(A,\lambda E)\bm{u}$ is equal to $\bm{u}$ itself. Therefore,
  if one wants to obtain the eigenvector which has a simple expression,
  it might be better to select a non-zero vector $\bm{u}$ from $V$. We
  arrive at the following strategy:
  \begin{enumerate}[label=(\roman{enumii})]
  \item select $j$ from $J_1$ as in (a) above;
  \item compute a basis $B$ of the Krylov vector space $L_A(\bm{u})$
    by column reductions;
  \item select a basis vector $\bm{u}$ from $B$ that has a simple form,
    or choose an appropriate one;
  \item compute the eigenvector $\varphi(\lambda)=\psi_p(A,\lambda
    E)\bm{u}$. 
  \end{enumerate}
\end{enumerate}

\section{Krylov vector space}
\label{sec:sec:eigenvect-multiplicity>1}

Let $f_p(\lambda)\in\Q[\lambda]$ be an eigenfactor of $A$, which
satisfies the condition $\max_{j\in J}\{l_{p,j}\}=1$. We give an
algorithm for computing the eigenvectors for all the roots
$\alpha_1,\alpha_2,\ldots,\alpha_d$ of $f_p(\lambda)$, where $d=d_p$
stands for the degree of $f_p$.

For $i=1,\ldots,d$, let $F_{p,\alpha_i}$ be the eigenspace associated
to the eigenvalue $\alpha_i$, and $F_p$ be the eigenspace associated
to the roots of $f_p(\lambda)=0$. Since the condition $\max_{j\in
  J}\{l_{p,j}\}=1$ implies $\dim(F_{p,\alpha_i})=m_p$, we have
$\dim_{\C}(F_p)=d_pm_p=d m_p$. The purpose of this section
is therefore to describe a method for computing $dm_p$ eigenvectors that
constitute a basis of the eigenspace $F_p$ associated to the
eigenfactor $f_p(\lambda)$.

Let $V=\Span{\bm{v}_{p,j}\mid j\in J_1}$, where $\bm{v}_{p,j}$ are
defined as in \cref{eq:vpj}.
For
$\bm{v}\in V$, let $L_A(\bm{v})$ be as in \cref{eq:lau}.
We have the following proposition.
\begin{proposition}
  \label{pro:eigenvector1}
  Let $\bm{u},\bm{w}\in V$ with $\bm{u},\bm{w}\ne\bm{0}$, and
  $\alpha_1,\ldots,\alpha_d$ be the roots of $f_p(\lambda)=0$. Then, the
  followings are equivalent:
  \begin{enumerate}[label={\upshape (\roman{enumi})}]
  \item $\Spanc{\psi_p(A,\alpha_i)\bm{u}\mid i=1,\ldots,d}
    =\Spanc{\psi_p(A,\alpha_i)\bm{w}\mid i=1,\ldots,d}$,
    \label{pro:eigenvector1:item1}
  \item $L_A(\bm{u})=L_A(\bm{w})$,\label{pro:eigenvector1:item2}
  \item $\bm{w}\in L_A(\bm{u})$,\label{pro:eigenvector1:item3}
  \item $\bm{u}\in L_A(\bm{w})$.\label{pro:eigenvector1:item4}
  \end{enumerate}
\end{proposition}
\begin{proof}
  Since $f_p(\lambda)$ is the minimal annihilating
  polynomial of $\bm{u}\in V$, $\bm{u}$, $A\bm{u}$, $A^2\bm{u}$, \ldots,
  $A^{d-1}\bm{u}$ are linearly independent. Furthermore, $A^k\bm{u}$ satisfies
  \begin{equation}
    \label{eq:scalar}
    \psi_p(A,\lambda E)(A^k\bm{u})=A^k(\psi_p(A,\alpha_i E)\bm{u})
    =\alpha_i^k(\psi_p(A,\alpha_i E)\bm{u}),
  \end{equation}
  which shows that \ref{pro:eigenvector1:item1} and
  \ref{pro:eigenvector1:item3} are equivalent.  Next, $\bm{w}$,
  $A\bm{w}$, $A^2\bm{w}$, \ldots, $A^{d-1}\bm{w}$ satisfy
  \[
  \psi_p(A,\alpha_i E)(A^k\bm{w})=A^k(\psi_p(A,\alpha_i E)\bm{w})
  =\alpha_i^k(\psi_p(A,\alpha E)\bm{w}),
  \]
  as in \cref{eq:scalar}. Since $\psi_p(A,\alpha E)\bm{u}$ is equal to
  $\psi_p(A,\alpha E)\bm{w}$ up to a scalar, we see that
  \ref{pro:eigenvector1:item3} and \ref{pro:eigenvector1:item4} are
  equivalent, thus we also have \ref{pro:eigenvector1:item2} is
  equivalent to the others, which completes the proof.
\end{proof}

Since there exist $m_p$ vectors 
$\bm{u}_1,\bm{u}_2,\ldots,\bm{u}_{m_p}\in V$ that satisfy
\[
V=L_A(\bm{u}_1)\oplus L_A(\bm{u}_2)\oplus\cdots\oplus
L_A(\bm{u}_{m_p}),
\]
we have
\begin{equation}
  \label{eq:fpspace}
  F_p=\Spanc{\psi_p(A,\lambda)\bm{u}_k\mid
  \lambda=\alpha_1,\alpha_2,\ldots,\alpha_d,\, k=1,\ldots,m_p}, 
\end{equation}
thus $\psi_p(A,\lambda)\bm{u}_k$ ($k=1,\ldots,m_p$) are the desired
eigenvectors.

Discussions above leads an algorithm for computing
eigenvectors; see Algorithm~\ref{alg:eigenvector-multiple-1}.

\begin{algorithm}
  \caption{Computing eigenvectors in the case that true unit minimal
    annihilating polynomials are given}
  \label{alg:eigenvector-multiple-1}
  \begin{algorithmic}[1]
    \Input{
      $A\in\Q^{n\times n}$; $f_p(\lambda)\in\Q[\lambda]$: an
      eigenfactor of $A$; $J_1\subset J$ satisfying \cref{eq:j1};
      $\{g_{p,j}(\lambda)\mid j\in J_1\}$: a set of cofactors, defined
      as in \cref{eq:gpj};
    }
    \Output{
      $\Phi=\{\varphi_1(\lambda),\ldots,\varphi_{m_p}(\lambda)\}$:
      the eigenvectors of $A$ associated to the root of
      $f_p(\lambda)=0$;  
    }
    \State $\Phi\gets\{\}$; $L\gets\{\}$;
    \For {$j\in J_{1}$} 
    $\bm{v}_j\gets g_{p,j}(A)\bm{e}_j$ with the Horner's rule
      (\cref{eq:fav});
    \EndFor 
    \State Calculate a basis $B$ of $V=\Span{\bm{v}_j\mid j\in
      J_{1}}$;
    \For {$k=1,\ldots,m_p-1$}
    \State Choose $\bm{u}\in B$ satisfying $\bm{u}\not\in L$
    which has the ``simplest'' form;
    \State
    Calculate
    $K_A(\bm{u})=\{\bm{u},A\bm{u},\ldots,A^{d-1}\bm{u}\}$; 
    \State
    $\varphi_k(\lambda)\gets$
    \textsc{CalculateEigenvector}($f_p(\lambda),K_A(\bm{u})$);
    \label{alg:eigenvector-multiple-1:varphik}
    \Comment See \Cref{rem:calculateeigenvector}
    \State Calculate a basis of $L_A(\bm{u})$ from
    $K_A(\bm{u})$ by the column reduction;
    \State $\Phi\gets\Phi\cup\{\varphi_k(\lambda)\}$;
    \State $L\gets L\oplus L_A(\bm{u})$;

    \EndFor
    \State Choose $\bm{u}\in B$ satisfying
    $\bm{u}\not\in L$ 
    which has the ``simplest'' form;
    \Comment Note that this step does not require
    computing $L_A(\bm{u})$, etc.
    \State $\varphi_{m_p}(\lambda)\gets\psi_p(A,\lambda
    E)\bm{u}$;
    \Comment Calculated using the Horner's rule (\cref{eq:fav}) 
    \State$\Phi\gets\Phi\cup\{\varphi_{m_p}(\lambda)\}$;

    \State\Return $\Phi$;
  \end{algorithmic}
\end{algorithm}

\begin{remark}
  \label{rem:calculateeigenvector}
  In Line
  \ref{alg:eigenvector-multiple-1:varphik} in
  \Cref{alg:eigenvector-multiple-1}, eigenvectors are computed using
  the Krylov vectors calculated in the preceding lines, as shown in
  Procedure \textsc{CalculateEigenvector} below.

  \noindent{\hrulefill}
  \begin{algorithmic}[1]
    \Input{
      $f_p(\lambda)=\lambda^{d}+a_{p,d-1}\lambda^{d-1}
      +\cdots+a_{p,0}\in\Q[\lambda]$: an eigenfactor of $A$ with $a_{p,d}=1$;
      $K_A$: a list of $d$ vectors of dimension $n$;
    }
    \Output{$\varphi(\lambda)=\psi_p(A,\lambda E)\bm{u}$: an
      eigenvector of $A$ associated to the root of
      $f_p(\lambda)=0$;}
    \Procedure{CalcuateEigenvector}{$f_p(\lambda)$, $K_A$}
    \For {$j=1,\ldots,d$} $\bm{c}_{d-j}\gets
    \sum_{k=0}^{j-1} a_{p,d-k}K_A[j-k]$;
    \Comment $K_A[i]$ denotes the $i$-th element in $K_A$
    \EndFor
    \State \Return $\lambda^{d-1}\bm{c}_{d-1}+
    \lambda^{d-2}\bm{c}_{d-2}\cdots+\lambda\bm{c}_1+\bm{c}_0 $;
    \EndProcedure
  \end{algorithmic}
  \hrulefill
\end{remark}

\begin{remark}
  In several lines in \Cref{alg:eigenvector-multiple-1}, 
  we take vectors of ``the simplest form'' from the certain set of
  vectors. ``The simplest form'' may be different
  according to different criteria, such as bit-length of the
  components, or the number of zero components in the calculated vectors.
\end{remark}

\section{Main results}
\label{sec:eigenvector-upap}

\Cref{alg:eigenvector-multiple-1} uses the minimum annihilating
polynomials effectively for computing eigenvectors. However, direct
use of the minimum annihilating polynomials often leads to relatively
high computational complexity.

In this section, the unit \emph{pseudo} annihilating polynomials
(\cite{taj-oha-ter2018}) are utilized for efficient computation of
eigenvectors. Pseudo annihilating polynomials are suitable for
computing eigenvectors in place of the minimal annihilating
polynomials because they coincide with high possibility. In addition,
computation of pseudo annihilating polynomials is more efficient than
that of the minimal annihilating polynomials.


First, we recall the notion of unit pseudo annihilating
polynomials from our previous paper (\cite{taj-oha-ter2018}).
Let $\bm{r}$ be a non-zero row vector over $\Z$ whose
components are randomly given. Let
\begin{equation}
  \label{eq:wpk}
  \begin{split}
    \bm{r}_p^{(0)} &=
    (r_{p,1}^{(0)},r_{p,2}^{(0)},\ldots,r_{p,n}^{(0)})=\bm{r}G_p,  
    \\
    \bm{r}_p^{(k)} &=
    (r_{p,1}^{(k)},r_{p,2}^{(k)},\ldots,r_{p,n}^{(k)})=\bm{r}G_p{F_p}^k
    \quad
    \textrm{for $k>0$,}
  \end{split}
\end{equation}
where $G_p=g_{p,j}(A)$ and $F_p=f_p(A)$.
Furthermore, for $j=1,\ldots,n$, define
\begin{equation}
  \label{eq:rhopj}
  \l'_{p,j}=
  \begin{cases}
    0 & \text{if $r_{p,j}^{(0)}=0$,}  \\
    k & \text{if $r_{p,j}^{(k-1)}\ne 0$ and $r_{p,j}^{(k)}=0$.}
  \end{cases}
\end{equation}
Consider the
polynomial $\pi'_{A,j}(\lambda)$ defined by 
\[
\pi'_{A,j}(\lambda)=f_1^{\l'_{1,j}}(\lambda)f_2^{\l'_{2,j}}(\lambda)\cdots
f_q^{\l'_{q,j}}(\lambda).
\]
We call $\pi'_{A,j}(\lambda)$ a $j$-th \emph{unit pseudo
annihilating polynomial} of $A$. Notice that $\pi'_{A,j}(\lambda)$
divides $\pi_{A,j}(\lambda)$.
Therefore, $\pi'_{A,j}(\lambda)=\pi_{A,j}(\lambda)$ if and only if
$\pi'_{A,j}(\lambda)\bm{e}_j=\pi_{A,j}(\lambda)\bm{e}_j$.
In the previous paper (\cite{taj-oha-ter2018}), an effective method for
computing $\pi'_{A,j}(\lambda)$ for $j\in J$ is given.

For $j\in J$, let 
\begin{equation}
  \label{eq:gpj'}
  g'_{p,j}(\lambda)
  =f_1(\lambda)^{l'_{1,j}}\cdots f_{p-1}(\lambda)^{l'_{p-1,j}}
  f_{p+1}(\lambda)^{l'_{p+1,j}}\cdots f_q(\lambda)^{l'_{q,j}},
\end{equation}
and
\begin{equation}
  \label{eq:j'1}
  J'_{1}=\{j\in J\mid l'_{p,j}=1\}.
\end{equation}



Next, we consider for $j\in J'_1$, two vectors $\bm{v}'_j$ and
$\varphi'(\lambda)$ defined to be $\bm{v}'_j=g'_{p,j}(A)\bm{e}_j$ and
$\varphi'(\lambda)=\psi_p(A,\lambda E)\bm{v}'_j$, respectively. Since
$f_p(A)=(A-\lambda E)\psi_p(A,\lambda E)$, we have
\begin{align*}
  \pi_{A,j}'(A)\bm{e}_j &= f_p(A)g_{p,j}'(A)\bm{e}_j \nonumber\\
  &= f_p(A)\bm{v}'_{p,j} \nonumber\\
  &= (A-\lambda E)\psi_p(A,\lambda E)\bm{v}'_{p,j}.
\end{align*}
Therefore,  $\pi'_{A,j}(\lambda)=\pi_{A,j}(\lambda)$ if and only if
$\varphi'(\lambda)=\psi_p(A,\lambda E)\bm{v}'$ is a true eigenvector
associated to the eigenfactor $f_p(\lambda)$. Furthermore, the
formula above shows that the calculation of $\varphi'(\lambda)$ is
contained in the calculation by the Horner's rule of
$\pi'_{A,j}(\lambda)\bm{e}_j$. More precisely, as we have mentioned in
\Cref{sec:prem}, $f_p(A)\bm{v}'_{j}=\pi'_{A,j}(A)\bm{e}_j$ is obtained
from $\varphi'(\lambda)=\psi_p(A,\lambda E)\bm{v}'_{j}$ by just one
last step of the Horner's rule:
\begin{equation}
  \label{eq:puap-test}
  A\bm{c}_0+a_0\bm{v}',
\end{equation}
where
$\varphi'(\lambda)=\lambda^{d-1}\bm{c}_{d-1}+\lambda^{d-2}\bm{c}_{d-2}+\cdots+\lambda\bm{c}_1+\bm{c}_0$.

Now, recall a method for computing the minimal annihilating
polynomials $\pi_{A,j}(\lambda)$, $j\in J$ proposed in
\cite{taj-oha-ter2018}. The method consists of mainly three
steps:
\begin{enumerate}[leftmargin=*, label={Step \arabic{enumi}.}]
\item Compute unit pseudo annihilating polynomials
  $\pi'_{A,j}(\lambda)$ for $j\in J$.
\item Compute $\pi'_{A,j}(\lambda)\bm{e}_j$ for $j\in J$ by
  the Horner's rule.
\item If $\pi_{A,j}'(\lambda)\bm{e}_j\ne 0$ for some $j$,
  then construct the minimal annihilating polynomial
  $\pi_{A,j}(\lambda)$ by computing the minimal annihilating
  polynomial of the vector $\pi'_{A,j}(\lambda)\bm{e}_j$.
\end{enumerate}

The discussion given in the present section shows that Step 2
involves the computation of a lot of eigenvectors. However, all the
information on eigenvalues are discarded by the direct use of the
Horner's rule. We conclude, in this regard, that Algorithm 1 which
utilizes the true unit minimal annihilating polynomials
$\pi_{A,j}(\lambda)$ for $j\in J_1$ has redundancy.


Now we are ready to design an efficient method for computing
eigenvectors associated to the eigenfactor $f_p(\lambda)$. Let
$l'_p=\max_{j\in J}\{l'_{p,j}\}$. Assume hereafter that $l'_p=1$ and
set
\[
J'_1=\{j\in J\mid l'_{p,j}=1\},\quad
J'_0=\{j\in J\mid l'_{p,j}=0\}.
\]
Let 
\[
G'=\{\bm{v}'_j=g'_{p,j}(\lambda)\bm{e}_j\mid j\in J'_1\},\quad
V'=\Span{\bm{v}'_j\mid j\in J'_1},
\]
$B'$ denotes a basis of the vector space $V'$.
We present two different algorithms. \Cref{alg:eigenvector-multiple-2}
uses the set $G'$ and \Cref{alg:eigenvector-multiple-3} uses the set
$B'$. \Cref{alg:eigenvector-multiple-2} is designed to compute
eigenvectors in an efficient manner. In contrast,
\Cref{alg:eigenvector-multiple-3} is designed with an intention of
obtaining simpler expression of eigenvectors.

\begin{algorithm}
  \caption{Computing eigenvectors with unit pseudo annihilating
    polynomials (for quick computation of eigenvectors)} 
  \label{alg:eigenvector-multiple-2}
  \begin{algorithmic}[1]
    \Input{
      $A\in\Q^{n\times n}$; $f_p(\lambda)\in\Q[\lambda]$: an
      eigenfactor of $A$; $J'_1\subset J$ satisfying \cref{eq:j'1};
      $\{g'_{p,j}(\lambda)\mid j\in J'_1\}$: a set of cofactors,
      defined as in \cref{eq:gpj'};
    }
    \Output{
      $\Phi=\{\varphi_1(\lambda),\ldots,\varphi_{m_p}(\lambda)\}$:
      the eigenvectors of $A$ associated to the root of
      $f_p(\lambda)=0$;
    }
    \State $\Phi\gets\{\}$; $G'\gets\{\}$; $L\gets\{\}$;
    \For {$j\in J'_{1}$} \label{alg:eigenvector-multiple-2:loop-jl}
    \label{alg:eigenvector-multiple-1:seeds-begin} 
    \State $\bm{v}'_j\gets g'_{p,j}(A)\bm{e}_j$ with the Horner's rule
    (\cref{eq:fav});
    \State $G'\gets G'\cup\{\bm{v}'_j\}$;
    \EndFor \label{alg:eigenvector-multiple-1:seeds-end}
    \For {$k=1,\ldots,m_p-1$}
    \State Choose $\bm{u}'\in G'$ satisfying $\bm{u}'\not\in L$
    which has the ``simplest'' form;
    \label{alg:eigenvector-multiple-2:setu'k}
    \State $G'\gets G'\setminus\{\bm{u}'\}$;
    \State Calculate
    $K_A(\bm{u}')=\{\bm{u}',A\bm{u}',\ldots,A^{d-1}\bm{u}'\}$;
    \State $\varphi_k(\lambda)\gets$
    \textsc{CalculateEigenvector}($f_p(\lambda),K_A(\bm{u}')$);
    \label{alg:eigenvector-multiple-2:eigenvector-k}
    \Comment See \Cref{rem:calculateeigenvector}
    \If{$f_p(A)\bm{u}'=\bm{0}$}  \label{alg:eigenvector-multiple-2:verify-k}
    \Comment Calculated as in \cref{eq:puap-test}
    \State Calculate a basis of $L_A(\bm{u}')$ from
    $K_A(\bm{u}')$ by the column reduction;
    \State $\Phi\gets\Phi\cup\{\varphi_k(\lambda)\}$;
    \State $L\gets L\oplus L_A(\bm{u}')$;
    \Else\mbox{} go to
      Line~\ref{alg:eigenvector-multiple-2:setu'k};
    \EndIf
    \EndFor
    \State Choose $\bm{u}'\in G'$ satisfying
    $\bm{u}'\not\in L$
    which has the ``simplest'' form;
    \label{alg:eigenvector-multiple-2:setu'mp}
    \Comment Note that this step does not require
    calculating $L_A(\bm{u}')$, etc.
    \State $G'\gets G'\setminus\{\bm{u}'\}$;
    \State $\varphi_{m_p}(\lambda)\gets\psi_p(A,\lambda
    E)\bm{u}'$; \label{alg:eigenvector-multiple-2:eigenvector-mp}
    \Comment Calculated using the Horner's rule (\cref{eq:fav}) 
    \If{$f_p(A)\bm{u}'=\bm{0}$}
    $\Phi\gets\Phi\cup\{\varphi_{m_p}(\lambda)\}$;
    \label{alg:eigenvector-multiple-2:verify-mp}
    \Comment Calculated as in \cref{eq:puap-test}
    \Else\mbox{}  go to
      Line~\ref{alg:eigenvector-multiple-2:setu'mp};
    \EndIf
    \State\Return $\Phi$;
  \end{algorithmic}
\end{algorithm}


\begin{algorithm}
  \caption{Computing eigenvectors with unit pseudo annihilating polynomials}
  \label{alg:eigenvector-multiple-3}
  \begin{algorithmic}[1]
    \Input{
      $A\in\Q^{n\times n}$; $f_p(\lambda)\in\Q[\lambda]$: an
      eigenfactor of $A$; $J'_1\subset J$ satisfying \cref{eq:j'1};
      $\{g'_{p,j}(\lambda)\mid j\in J'_1\}$: a set of cofactors,
      defined as in \cref{eq:gpj'};
    }
    \Output{
      $\Phi=\{\varphi_1(\lambda),\ldots,\varphi_{m_p}(\lambda)\}$:
      the eigenvectors of $A$ associated to the root of
      $f_p(\lambda)=0$;
    }
    \State $\Phi\gets\{\}$; $L\gets\{\}$;
    \For {$j\in J'_{1}$} \label{alg:eigenvector-multiple-3:loop-jl}
    \State $\bm{v}'_j\gets g'_{p,j}(A)\bm{e}_j$ with the Horner's rule
    (\cref{eq:fav});
    \State $c_j\gets (\mbox{a random integer})$;
    \EndFor
    \State $\bm{v}'\gets\sum_{j\in J'_1}c_j\bm{v}'_j$;
    \If {$f_p(A)\bm{v}'\ne\bm{0}$}
    \label{alg:eigenvector-multiple-3:verify-0}
    exit with an error message: ``One or more pseudo
    annihilating polynomial(s) are wrong''; 
    \label{alg:eigenvector-multiple-3:error-1}
    \EndIf
    \State Calculate $B'$ as a basis of $V'=\Span{\bm{v}'_j\mid j\in
      J'_{1}}$; 
    \label{alg:eigenvector-multiple-3:B'}
    \For {$k=1,\ldots,m_p-1$}
    \State Choose $\bm{u}'\in B'$ satisfying $\bm{u}'\not\in L$
    which has the ``simplest'' form;
    \label{alg:eigenvector-multiple-3:setu'k}
    \State $B'\gets B'\setminus\{\bm{u}'\}$;
    \State Calculate
    $K_A(\bm{u}')=\{\bm{u}',A\bm{u}',\ldots,A^{d-1}\bm{u}'\}$;
    \State $\varphi_k(\lambda)\gets$
    \textsc{CalculateEigenvector}($f_p(\lambda),K_A(\bm{u}')$);
    \label{alg:eigenvector-multiple-3:eigenvector-k}
    \Comment See \Cref{rem:calculateeigenvector}
    \If{$f_p(A)\bm{u}'=\bm{0}$}  \label{alg:eigenvector-multiple-3:verify-k}
    \Comment Calculated as in \cref{eq:puap-test}
    \State Calculate a basis of $L_A(\bm{u}')$ from
    $K_A(\bm{u}')$ by the column reduction;
    \State $\Phi\gets\Phi\cup\{\varphi_k(\lambda)\}$;
    \State $L\gets L\oplus L_A(\bm{u}')$;
    \Else \mbox{}
    exit with an error message: ``One or more pseudo
    annihilating polynomial(s) are wrong''; 
    \EndIf
    \EndFor
    \algstore{alg:eigenvector-multiple-3}
  \end{algorithmic}
\end{algorithm}
\begin{algorithm}
  \ContinuedFloat
  \caption{Computing eigenvectors with unit pseudo annihilating
    polynomials (Continued)}
  \begin{algorithmic}
    \algrestore{alg:eigenvector-multiple-3}
    \State Choose $\bm{u}'\in B'$ satisfying
    $\bm{u}'\not\in L$
    which has the ``simplest'' form;
    \label{alg:eigenvector-multiple-3:setu'mp}
    \Comment Note that this step does not require
    calculating $L_A(\bm{u}')$, etc.
    \State $B'\gets B'\setminus\{\bm{u}'\}$;
    \State $\varphi_{m_p}(\lambda)\gets\psi_p(A,\lambda
    E)\bm{u}'$; \label{alg:eigenvector-multiple-3:eigenvector-mp}
    \Comment Calculated using the Horner's rule (\cref{eq:fav}) 
    \If{$f_p(A)\bm{u}'=\bm{0}$}
    $\Phi\gets\Phi\cup\{\varphi_{m_p}(\lambda)\}$;
    \label{alg:eigenvector-multiple-3:verify-mp}
    \Comment Calculated as in \cref{eq:puap-test}
    \Else \mbox{}
    exit with an error message: ``One or more pseudo
    annihilating polynomial(s) are wrong''; 
    \EndIf
    \State\Return $\Phi$;
  \end{algorithmic}
\end{algorithm}

\begin{remark}
  In \Cref{alg:eigenvector-multiple-2},
  $\bm{v}'_j=g'_{p,j}(A)\bm{e}_j$ are directly used for efficient
  construction of candidates of eigenvectors. Furthermore, in the case
  that $\varphi'(\lambda)=\psi_p(A,\lambda E)\bm{v}'_j$ is not a true
  eigenvector, another candidate is computed immediately just by
  picking up $\bm{v}'_{j'}\in G'$ with $j'\ne j$. We continue to pick
  up new $\bm{v}'_j$ until $m_p$ eigenvectors are computed.
\end{remark}

\begin{remark}
  In \Cref{alg:eigenvector-multiple-3}, if there is a vector
  $\bm{u}'\in B'$ which does not satisfy the condition
  $f_p(A)\bm{u}'=\bm{0}$, there may exist many such vectors, because
  $B'$ is calculated from $G'$ by column reduction.  Therefore, in the
  case when such vector is detected, we recalculate pseudo annihilating
  polynomials of $A$ and start over computation of
  \Cref{alg:eigenvector-multiple-3}.
\end{remark}

\begin{remark}
  In both algorithms, it is sufficient to verify
  $f_p(A)\bm{u}'_k=\bm{0}$ only for vectors 
  $\bm{u}'_{1},\ldots,\bm{u}'_{m_p}$ in the basis
  $V'=\Span{\bm{v}'_j\mid j\in J'_1}$.
  This reduces the cost of computation considerably.
\end{remark}




\section{Experiments}
\label{sec:exp}

We have implemented
\Cref{alg:eigenvector-multiple-2,alg:eigenvector-multiple-3} on 
a computer algebra system Risa/Asir (\cite{nor2003})
and evaluated them. First, for the case of $m_p=1$, we have
computed eigenvectors with changing $\dim(A)$ and
$\deg(\pi'_{A,j}(\lambda))$. Second, we have computed
eigenvectors for the case $m_p=2,3,4$  with focusing attention on calculation
and reduction of ``seeds'' of eigenvectors. 
Finally, we have compared performance of our algorithms with
an algorithm implemented on Maple (\cite{maple2016}). 

The tests were carried out on the following environment: Intel Xeon
E5-2690 at 2.90 GHz, RAM 128GB, Linux 2.6.32 (SMP).

\subsection{Computing eigenvectors with $m_p=1$}
\label{sec:exp-1}


In this experiment, test matrices are given as follows. Let
$f_1(x),\ldots,f_8(x)$ be monic and pairwise relatively prime
polynomials of the same degree.  For
$\bar{A}=\diag(C(f_1),C(f_2),\ldots,C(f_8))$,
where $C(f)$ denotes the companion matrix of $f$, we have
calculated dense test matrix $A$
by applying similarity transformations. 



\Cref{tab:test-1-max,tab:test-1-min} show the results with changing
dimension of the matrix. In the amount of memory usage, ``$a$e$b$''
denotes $a\times10^{b}$ (bytes). In \Cref{tab:test-1-max},
eigenvectors are computed with pseudo annihilating polynomials of
degree $\deg(\pi_{A,j}(\lambda))=\dim(A)$. On the other hand, in
\Cref{tab:test-1-min}, eigenvectors are computed with pseudo
annihilating polynomials of degree
$\deg(\pi_{A,j}(\lambda))=\dim(A)/4$. 


\begin{table}[p]
  \centering
  \caption{Computing time and memory usage for the case of
    $\deg(\pi_{A,j}(\lambda))=\dim(A)$.
    See~\Cref{sec:exp-1} for details.}
  \begin{tabular}{r|r|r|r}
    \hline
    $\dim(A)$ & $\deg(\pi_{A,j})$ & Time (sec.) & Memory usage (bytes) \\
    \hline
    128 & 128 & 0.205 & 2.37e8 \\
    256 & 256 & 2.037 & 2.11e9 \\
    384 & 384 & 8.971 & 8.76e9 \\
    512 & 512 & 29.57 & 2.61e10 \\
    640 & 640 & 50.48 & 4.37e10 \\
    768 & 768 & 105.58 & 9.22e10 \\
    896 & 896 & 164.75 & 1.50e11 \\
    1024 & 1024 & 289.72 & 2.57e11 \\
    \hline
  \end{tabular}
  \label{tab:test-1-max}
\end{table}
\begin{table}[p]
  \centering
  \caption{Computing time and memory usage for the case of
    $\deg(\pi_{A,j}(\lambda))=\dim(A)/4$.
    See~\Cref{sec:exp-1} for details.}
  \begin{tabular}{r|r|r|r}
    \hline
    $\dim(A)$ & $\deg(\pi_{A,j})$ & Time (sec.) & Memory usage (bytes) \\
    \hline
    128 & 32 & 0.033 & 4.50e7 \\
    256 & 64 & 0.322 & 3.78e8 \\
    384 & 96 & 1.379 & 1.47e9 \\
    512 & 128 & 3.598 & 3.55e9 \\
    640 & 160 & 5.471 & 5.54e9 \\
    768 & 192 & 12.88 & 1.13e10 \\
    896 & 224 & 21.19 & 1.78e10 \\
    1024 & 256 & 34.98 & 3.01e10 \\
    \hline
  \end{tabular}
  \label{tab:test-1-min}
\end{table}


\Cref{tab:test-2-min,tab:test-2-max}
show the results using $\pi_{A,j}(\lambda)$ of different degrees for
the same matrix.

From both experiments, we see that eigenvectors are computed more efficiently by using
pseudo annihilating polynomials of smaller degrees.


\begin{table}[p]
  \caption{Computing time and memory usage for the case of
    $\dim(A)=128$ with increasing the degree of the minimal
    annihilating polynomial. See \Cref{sec:exp-1} for details.}
  \centering
  \begin{tabular}{r|r|r}
    \hline
    $\deg(\pi_{A,j})$ & Time (sec.) & Memory usage (bytes) \\
    \hline
    32 & 0.033 & 4.50e7 \\
    48 & 0.054 & 7.47e7 \\
    64 & 0.085 & 1.05e8 \\
    80 & 0.109 & 1.37e8 \\
    96 & 0.144 & 1.73e8 \\
    112 & 0.170 & 2.04e8 \\
    128 & 0.204 & 2.37e8 \\
    \hline
  \end{tabular}
  \label{tab:test-2-min}
\end{table}
\begin{table}[p]
  \caption{Computing time and memory usage for the case of
    $\dim(A)=1024$ with increasing the degree of the minimal
    annihilating polynomial. See \Cref{sec:exp-1} for details.}
  \centering
  \begin{tabular}{r|r|r}
    \hline
    $\deg(\pi_{A,j})$ & Time (sec.) & Memory usage (bytes) \\
    \hline
    256 & 34.98 & 3.01e10 \\
    384 & 61.80 & 5.29e10 \\
    512 & 95.06 & 7.98e10 \\
    640 & 135.71 & 1.16e11 \\
    768 & 172.33 & 1.54e11 \\
    896 & 222.82 & 2.02e11 \\
    1024 & 289.72 & 2.57e11 \\
    \hline
  \end{tabular}
  \label{tab:test-2-max}
\end{table}



\subsection{Computing eigenvectors with $m_p>1$}
\label{sec:exp-2}

In this experiment, test matrices are given in the same way as
above. In test matrices, the number of $\pi_{A,j}(\lambda)$ with
$l_{p,j}=1$ is approximately equal to $\dim(A)/4$. Among them,
approximately half of them have degree $\dim(A)/4$, the other half
have degree $\dim(A)$. For each cases, the same test matrices are used.

\Cref{tab:test-multiple-1} shows the results for
\Cref{alg:eigenvector-multiple-2}. ``Time ($G'$)'' denotes time for
  computing $G'$, the set of ``seeds'' of eigenvectors (lines
  \ref{alg:eigenvector-multiple-1:seeds-begin}--\ref{alg:eigenvector-multiple-1:seeds-end}). ``\#$G'$'' 
  denotes the number of elements in $G'$.
  Although computing time of $G'$ is long for large $A$, it can be
  reduced by the use of parallel processing
  (e.g.\ \cite{mae-nor-oha-tak-tam2001}) since all the
  vectors in $G'$ can be calculated independently with the Horner's rule.

\Cref{tab:test-multiple-2} shows the results for
\Cref{alg:eigenvector-multiple-3}. ``Time ($B'$)'' denotes computing
time of $B'$ (line
\ref{alg:eigenvector-multiple-3:B'}) from construction of $G'$.

In \Cref{tab:test-multiple-3}, ``$\max\{\|\varphi_2(\lambda)\|_2\}$'' and
``$\max\{\|\varphi_3(\lambda)\|_2\}$'' denote the maximum values of the 2-norms
of eigenvectors computed by
\Cref{alg:eigenvector-multiple-2,alg:eigenvector-multiple-3},
respectively. Notice that, in \Cref{alg:eigenvector-multiple-3}, the
norm of computed eigenvectors has remarkably decreased.


\begin{table}[p]
  \caption{Computing time and memory usage of
    \Cref{alg:eigenvector-multiple-2} for the case of $m_p>1$. See 
    \Cref{sec:exp-2} for details.}
  \centering
  \begin{tabular}{r|r|r|r|r|r|r}
    \hline
    $\dim(A)$ & $\deg(f_p)$ & $m_p$ & Time (sec.) & Memory usage &
    Time ($G'$) & \#$G'$ \\
    \hline
    128 & 4 & 2 & 4.072 & 4.61e9 & 4.008 & 28\\
    128 & 4 & 3 & 3.992 & 4.44e9 & 3.860 & 32\\
    128 & 4 & 4 & 4.304 & 4.48e9 & 3.780 & 32\\
    256 & 8 & 2 & 77.21 & 8.14e10 & 75.60 & 64\\
    256 & 8 & 3 & 71.12 & 7.00e10 & 65.33 & 64\\
    256 & 8 & 4 & 89.37 & 7.86e10 & 76.52 & 64\\
    512 & 16 & 2 & 1819.1 & 1.79e12 & 1797.61 & 128\\
    512 & 16 & 3 & 1319.6 & 1.21e12 & 1243.96 & 128\\
    512 & 16 & 4 & 2302.3 & 1.75e12 & 1780.29 & 128\\
    \hline
  \end{tabular}
  \label{tab:test-multiple-1}
\end{table}

\begin{table}[p]
  \caption{Computing time and memory usage of
    \Cref{alg:eigenvector-multiple-3} for the case of $m_p>1$. See 
    \Cref{sec:exp-2} for details.}
  \centering
  \begin{tabular}{r|r|r|r|r|r}
    \hline
    $\dim(A)$ & $\deg(f_p)$ & $m_p$ & Time (sec.) &
    Memory usage & Time ($B'$) \\
    \hline
    128 & 4 & 2 & 4.156 & 4.74e9 & 0.104 \\
    128 & 4 & 3 & 4.232 & 4.61e9 & 0.204 \\
    128 & 4 & 4 & 4.192 & 4.45e9 & 0.100 \\
    256 & 8 & 2 & 77.81 & 8.12e10 & 0.640 \\
    256 & 8 & 3 & 69.14 & 6.98e10 & 1.308 \\
    256 & 8 & 4 & 85.86 & 7.86e10 & 2.608 \\
    512 & 16 & 2 & 1840.7 & 1.79e12 & 9.376 \\
    512 & 16 & 3 & 1355.7 & 1.21e12 & 27.34 \\
    512 & 16 & 4 & 2253.5 & 1.76e12 & 112.94 \\
    \hline
  \end{tabular}
  \label{tab:test-multiple-2}
\end{table}

\begin{table}[p]
  \caption{The maximum value of 2-norms of eigenvectors computed by
    \Cref{alg:eigenvector-multiple-2} ($\varphi_2(\lambda)$) and
    \Cref{alg:eigenvector-multiple-3} ($\varphi_3(\lambda)$). See
    \Cref{sec:exp-2} for details.} 
  \centering
  \begin{tabular}{r|r|r|r|r}
    \hline
    $\dim(A)$ & $\deg(f_p)$ & $m_p$ & 
    $\max\{\|\varphi_2(\lambda)\|_2\}$ & $\max\{\|\varphi_3(\lambda)\|_2\}$ \\
    \hline
    128 & 4 & 2 & 4.03e24 & 2.38e3\\
    128 & 4 & 3 & 6.64e8 & 1.03e3\\
    128 & 4 & 4 & 2.78e10 & 1.34e2\\
    256 & 8 & 2 & 4.75e15 & 8.76e1\\
    256 & 8 & 3 & 6.90e15 & 5.66e1\\
    256 & 8 & 4 & 2.05e29 & 1.01e2\\
    512 & 16 & 2 & 1.18e27 & 1.17e2\\
    512 & 16 & 3 & 1.40e26 & 1.06e2\\
    512 & 16 & 4 & 6.61e49 & 1.01e2\\
    \hline
  \end{tabular}
  \label{tab:test-multiple-3}
\end{table}

\subsection{Comparison of performance with Maple}
\label{sec:exp-3}

In this experiment, Test matrices are given as $A=(a_{ij})$ with integers
$a_{ij}$ satisfying $|a_{ij}|<10$ and $\dim(A)=8s$ with
$s=1,2,\dots,7$.  We have executed
``\texttt{LinearAlgebra:-Eigenvectors}'' function with
``\texttt{implicit=true}'' option for expressing eigenvalues as the
characteristic polynomial.  In each degree, we have measured computing
time and memory usage for computing eigenvectors of the same matrix
for $5$ times and have taken the average.

Table~\ref{tab:test-4} shows the results with computing time
in seconds and memory usage in bytes. Furthermore,
since Maple calculates the characteristic polynomial of the matrix,
we have measured computing time for calculating characteristic
polynomial $\chi_A(\lambda)$ of the given matrices independently,
which is shown in the rightmost column in the table. We see that, in
each dimension of $A$, computing time for the characteristic
polynomial accounts only a small portion of computing time for
eigenvectors. This result demonstrates efficiency of our method.

\begin{table}
  \caption{Computing time and memory usage by Maple. See
    \Cref{sec:exp-3} for details.}
  \centering
  \begin{tabular}{r|r|r|r}
    \hline
    $\dim(A)$ & Time (sec.) & Memory usage & Time for $\chi_A(\lambda)$ \\
    \hline
    8 & 0.24 & 8.40e6 & 4.8e$-3$\\
    16 & 9.40 & 7.68e7 & 5.8e$-3$\\
    24 & 146.80 & 1.26e8 & 7.2e$-3$\\
    32 & 2128.74 & 3.14e8 & 7.4e$-3$\\
    40 & 21584.16 & 2.08e9 & 1.6e$-3$\\
    48 & 41478.60 & 1.64e11 & 1.28e$-2$\\
    56 & 159304.81 & 2.89e11 & 3.12e$-2$\\
    \hline
  \end{tabular}
  \label{tab:test-4}
\end{table}

\section{Concluding remarks}
\label{sec:concl}

In this paper, we have proposed efficient algorithms for computing
eigenvector of matrices of integers under the assumption that the
geometric multiplicity of the eigenvalue is equal to the algebraic
multiplicity. The resulting algorithms utilize pseudo unit
annihilating polynomials, the Horner's rule for matrix polynomial
with vectors and Krylov vector spaces in an efficient manner.


The results of experiments show high performance of the resulting
algorithms. 

Based on the concept of (pseudo) annihilating polynomials, the first
and the second authors of the present paper studied a method for
computing generalized eigenvectors and reported basic ideas
(\cite{oha-taj2014}, \cite{oha-taj2015}, \cite{taj2013}). Algorithms
for computing generalized eigenvectors will be described in
forthcoming papers.

\def\cprime{$'$}



\end{document}